 \documentclass[letterpaper, 10 pt, conference]{IEEEconf}

\IEEEoverridecommandlockouts                           
\usepackage{graphicx}  

\usepackage{subcaption}
\usepackage{balance}

\usepackage{amsmath,amsfonts,amssymb,amscd}
\usepackage[ruled,vlined]{algorithm2e} 
\usepackage{mathtools,cleveref}

\usepackage{graphicx}
\usepackage{color}
\usepackage{pdfsync}

\usepackage{tikz, pgfplots, circuitikz}
\usetikzlibrary{arrows,automata,positioning,shapes,intersections,calc}
\pgfplotsset{compat=newest}

\newlength\figureheight
\newlength\figurewidth

\newtheorem{remark}{\bfseries Remark}

\newtheorem{prop}{\bfseries Proposition}

\DeclareMathOperator{\argmin}{argmin}

\usepackage{url,changebar,bm,xspace,dsfont}
\let\mathbb=\mathds 

\DeclareMathOperator{\Tr}{Tr}

\title{ \bf  An Online Feedback Optimization Approach to Voltage Regulation in Inverter-Based Power Distribution Networks}

\author{Alejandro D. Dom\'inguez-Garc\'ia,  Madi Zholbaryssov, Temitope Amuda, and Olaoluwapo Ajala
\thanks{Dom\'inguez-Garc\'ia, Amuda, and Ajala are with the Department of Electrical and Computer Engineering of the University of Illinois at Urbana-Champaign, Urbana, IL 61801,
USA. Email: {\tt\small \{aledan, tamuda2, ooajala2\}@ILLINOIS.EDU.}
}
\thanks{Zholbaryssov is with Typhoon HIL Inc., Somerville, MA 02143, USA. Email: Email: {\tt\small madi.zholbaryssov@typhoon-hil.com}. }
\thanks{This work was supported in part by the U.S. Department of Energy’s Office of Energy Efficiency and Renewable Energy (EERE) under Solar Energy Technologies Office (SETO) Agreement Number EE0009025; and the C3.ai Digital Transformation Institute.} 
}

\begin{document}

\maketitle

\thispagestyle{empty}
\pagestyle{empty}

\begin{abstract}
We address the  problem of controlling the reactive power setpoints of a set of distributed energy resources (DERs) in a power distribution network so as to mitigate the impact of variability in uncontrolled power injections associated with, e.g., renewable-based generation.  We formulate the control design problem as a stochastic optimization problem, which we solve online using a modified version of a  projected stochastic gradient descent (PSGD) algorithm. The proposed PSGD-based algorithm utilizes sensitivities of changes in  bus voltage magnitudes to changes in DER reactive power setpoints; such sensitivities are learned online via a recursive least squares estimator (rLSE). To ensure proper operation of the rLSE, the sequence of incremental changes in DER reactive power setpoints needs to be persistently exciting, which is guaranteed by a mechanism  built into the   controller. We analyze the stability of the closed-loop system and showcase controller performance via numerical simulations on the IEEE  123-bus distribution test feeder.
\end{abstract}



\section{Introduction}
\label{sec:intro}
Driven by a global effort to decarbonize the electricity sector, electric power distribution networks are undergoing major transformations in both power production and demand. These transformations include the massive integration of variable generation, e.g., photovoltaic   installations, new types of loads, e.g., plug-in electric vehicles, and storage devices; such energy assets are commonly referred to as distributed energy resources (DERs). Coupled with more frequent (and possibly random) network topology changes, and the increasing uncertainty in load demand, the rapidly increasing deployment of DERs poses numerous operational challenges for which existing control schemes in power distribution networks are not well equipped to handle. In light of this, the objective of this paper is to address one such challenge---ensuring effective regulation of voltage magnitude across all buses of the power distribution  network. 

 Currently, voltage   regulation in power distribution networks is accomplished, for the most part, through tap-changing under-load transformers  (TCULs) and fixed/switched capacitor banks (see, e.g., \cite[p.~16]{IEEE_std_voltage2018}). However, while these devices are effective in managing slow changes in voltage (minutes to hours), they are not suitable for managing fast voltage fluctuations (seconds to minutes) arising, e.g., from rapid changes in renewable-based power generation. This problem can be effectively addressed by controlling the reactive power injected into the distribution network by power-electronic inverter-interfaced DERs---a solution that has been actively pursued in the last decade   (see, e.g, \cite{BrettADG2013,XuDo20} and the references therein).

 
 In this paper, we also pursue the idea of utilizing inverter-interfaced DERs for voltage regulation. Building  on our earlier work on data-driven control algorithms in  \cite{XuDo20,XuDo19,MadiADG2021}, we design a voltage regulation  scheme that does not rely on an a priori known model of the system to be controlled; instead, the scheme utilizes data to estimate such a model online, while simultaneosly executing a feedback control algorithm.  We demonstrate that the proposed approach is adaptive to constantly varying system conditions and disturbances, and is capable of utilizing the reactive power support capabilities of inverter-based resources to effectively provide voltage regulation at a fast time-scale (i.e., milliseconds to seconds).
 
 The control design problem is cast as a stochastic optimization problem whose goal is to determine the reactive power setpoints of DERs so as to minimize the expectation of bus voltage deviations from their nominal values. In order to solve this problem, we utilize a projected stochastic gradient descent (PGSD) algorithm (see e.g., \cite{kushner2003stochastic}). Proper initialization and execution of the algorithm for a single step essentially results in a feedback controller that utilizes measured voltage deviations to adjust reactive power setpoints of DERs. In its basic form, the algorithm relies on knowing the sensitivities of changes in bus voltages with respect to changes in reactive power injections. Instead of obtaining these sensitivities offline via a model, we design a recursive least squares estimator (rLSE) that learns them online by using real-time voltage measurement data and the sequence of reactive power setpoints generated by the controller. The rLSE is executed in parallel with the controller, and in order to ensure  proper operation of the estimator, the sequence of incremental changes in DER reactive power setpoints needs to be persistently exciting. To this end, we modify the basic control algorithm  to include a mechanism that ensures this. 
 
 Most of the existing literature on utilization of inverter-interfaced DERs for voltage regulation rely on the use of exact models of the network (see, e.g., \cite{BrettADG2013} and the references therein). However, these methodologies have practical limitations in that, due to the limited number of sensors in power distribution networks, the models are hard to obtain in practice. More recently, there have been several papers that propose the use of data-driven techniques for addressing the voltage regulation  problem in power distribution networks, as well as in other related control and estimation problems.  For example, the authors of \cite{XuDo19} present a data-driven framework for coordinating the active and reactive power injections of DERs to provide  voltage regulation in radial networks. The proposed scheme utilizes estimates of network voltage sensitivities obtained online by fitting measurements to the so-called LinDistFlow model; by contrast, the method we propose  here is not restricted to radial networks, and it does not assume any particular model structure. The authors of \cite{XuDo20} propose a data-driven voltage regulation approach based on the estimated topology and line parameters of radial power distribution systems. The authors of  \cite{MadiADG2021} propose a model-free control scheme for regulating voltage, frequency, and line flows  that assumes no prior information on the system, and, hence, is highly adaptable to intermittent operating conditions. Additionally, a number of papers (see, e.g., \cite{GoTo19,MaKa20,NaLi2021}) have proposed data-driven approaches that use measurement data to directly perform controller synthesis without system identification.  
The use of linear estimation techniques to learn sensitivities of various power system state variables to control inputs has also  been exploited in numerous applications in bulk power system monitoring and control   \cite{ChDo14,HoDo16,ZhDh20}.

 The framework proposed in this paper is closely related to that presented in \cite{Pic22}, which  proposes a model-free, real-time optimal power flow solver through feedback optimization. Similar  to the approach we adopt here,   \cite{Pic22} proposes  a  data-driven online approach to learn  a  model of input–output sensitivities. The objective function of the online feedback optimization formulated is the operational cost of the inputs, and the inequality constraints are the operational limitations on the inputs. The sensitivity estimate at each iteration is computed using both the previous estimate and the current measurements only, and  a projected gradient solution method for the optimization problem, with persistent excitation of the control inputs performed at each iteration. By contrast, in our work, we  formulate the control problem  as a stochastic optimization problem, where the objective function is the expected value of the norm of the deviation of the outputs from their nominal values. In addition, the sensitivity estimate in our work is computed using all previous measurements with less weights assigned to older measurements. Finally, while the proposed scheme in \cite{Pic22} does not  account for the constraints on  control inputs when  attempting to generate persistently exciting inputs,   our work does take this important issue into account.

The remainder of this paper is organized as follows. We begin by formulating the voltage regulation problem in Section~\ref{sec:problemFormulation}. In order to solve such problem, Section~\ref{MBFO} describes a  model-based feedback optimization controller that relies on  sensitivities of voltage deviations to reactive power injections; such sensitivities are learnt online via a rLSE estimator described  in \Cref{sec:onlineEstimator}. Then, our proposed   controller, which is based on combining the ideas in Section~\ref{MBFO} and   \Cref{sec:onlineEstimator}, is presented in   \Cref{sec:feedbackOptimization}.  In \Cref{sec:stabilityAnalysis}, we present the main results of our convergence   analysis, whereas in   \Cref{sec:numericalResults}, we provide numerical simulation results  that demonstrate the effectiveness of our control scheme. Concluding remarks are given in \Cref{sec:concludingRemarks}.

\section{Problem Formulation}
\label{sec:problemFormulation}
Consider a power distribution network with $n+1$ buses indexed by the elements in $\mathcal{N}=\{0,1,2,\dots,n\}$, where the~$0$ element corresponds to  the bus at which the network is connected to an external system, e.g., a sub-transmission grid. Assume the network has $m$ reactive-power-capable DERs indexed by the elements in $\mathcal{C}=\{1,2\dots,m\}$. Let $v_i(t)$ denote the magnitude of the phasor associated with the voltage at bus~$i,~i \in \mathcal{N}$, at time $t$. Assume that $v_0(t)=V_0$, where $V_0$ is a positive constant, for all $t \geq 0$, and define $v(t):=\big[v_1(t),v_2(t),\dots,v_n(t) \big]^\top$. Also, let $q_i(t)$ denote the reactive power injected into the network by reactive-power capable DER~$i,~i \in \mathcal{C}=\{1,2\dots,m\}$, at time $t$, and define $q(t):=\big[q_1(t),q_2(t),\dots,q_m(t) \big ]^\top$.

Assume that, initially, $q(t_0)=q_0$, where $q_0 \in \mathbb{R}^m$ is given. Then, the objective   is to adjust $q(t_k)=:q_{k},~k \geq 1,$ via a feedback controller so as to regulate $v(t)$ to some  $v^*=\big[v_1^*, v_2^*, \dots, v_n^* \big]^\top,$ where $v_i^*>0,~i=1,2,\dots,n,$ denotes some nominal value. Let $\{t_k^+\}_{k \geq 1}$, $t_k < t_k^+ < t_{k+1}$, denote the sequence of time instants at which $v(t)$ is measured, and define $v_{k}:=v(t_k^+)$.   Then, we have that
\begin{align}
v_{k}=h\big(q_{k}, w_{k}\big), \label{input_output_model}
\end{align}
where $w_{k} =\big[w_{1,k},w_{2,k},\dots,w_{d,k} \big]^\top$, with {$w_{j,k}:=w_j(t_k),~j=1,2,\dots,d,$} representing  exogenous disturbances associated with, e.g., DER active power generation, load active and reactive power demand, and  network parameters.
\begin{remark} \label{remark_1}
In general,  it is difficult to analytically characterize the function  $h(\cdot,\cdot)$ as this essentially entails obtaining an analytical solution to the power flow problem, which is defined by a set of nonlinear equations describing the active and reactive power balance at each bus (see, e.g., \cite[pp.~323-330]{bergen2000power}). To be more specific, these equations  map the magnitudes and phase angles of the phasors associated with bus voltages to the active and reactive power injections at all buses; hence one would need to invert such mapping to analytically characterize $h(\cdot,\cdot)$.
{Therefore, in the remainder we will assume that $h(\cdot,\cdot)$ is not known.}\end{remark}

Assume that values taken by $w_k,~k \geq 1$, are of the form
\begin{align}
    w_k=w_k^\circ+\xi_k,
\end{align}
with the value taken by  $w_k^\circ$ known and  slowly changing as~$k$ evolves, and where  $\xi_k$ is not a priori known but can  be  described by a random vector $\Xi_k \in \mathbb{R}^d$.  In the remainder, we assume that the random vectors $\Xi_k,~k \geq 0,$ are independent and identically distributed (i.i.d.) with zero mean,  i.e., 
\begin{align}
    \text{E}\big[\Xi_k\big]=0,\quad k \geq 1,
\end{align}
and whose covariance  matrix is such that  
\begin{align}
    \text{E} \big[\|\Xi_k \|_2^2 \big] \leq \sigma,
\end{align}
where $\sigma$ denotes some positive constant. Now, since the values taken by $w_k^\circ,~k \geq 1$, are known, we can equivalently describe the relation in \eqref{input_output_model} as follows:
\begin{align}
    v_k=h_k(q_k,\xi_k), \label{input_output_model_2}
\end{align}
where {$h_k(q_k,\xi_k) \coloneqq h(q_{k}, w_k^\circ+\xi_k)$}.
 In the remainder, we will assume $q_0$ is such that
\begin{align}
v_0&=h(q_0,w_0^\circ) \nonumber \\
& =  v^*,
\end{align}
with the value taken by $w_0^\circ$ known, 
i.e., initially there is no uncertainty in the value that the exogenous disturbances takes, and the reactive powers injected by the DERs  are such that the magnitudes of all bus voltages  are equal to their nominal values.


\section{Model-Based Feedback Optimization} \label{MBFO}

In order to determine the value of $q_{k},~k \geq 1$,  consider  the following optimization problem:
\begin{subequations}  \label{model_based_opt}
\begin{align} 
\underset{\varphi_{k}}{\mbox{minimize}} & \quad \frac{1}{2} \text{E}_{\Xi_k} \bigg [\big \| v^*  - h_k\big(\varphi_{k},\Xi_{k}\big) \big \|_2^2  \bigg ]    \\
\mbox{subject to} & \quad \underline{ q}_{k} \leq \varphi \leq \overline{ q}_{k},
\end{align}
\end{subequations}
where $\text{E}_{\Xi_k}[\cdot]$ denotes  expectation over the distribution of $\Xi_k${, and $\underline{ q}_{k}$ and $\overline{ q}_{k}$ denote the lower and upper bounds, respectively, on the reactive power that can be injected by DERs into the network.}
Then,  a  local minimum of \eqref{model_based_opt}, which we denote by  $\varphi^*_{k}$, can be approximately obtained as $\varphi^*_{k} \approx \lim_{r \to \infty} \varphi^r_{k}$, with the evolution of $\varphi^r_{k},$ governed by
\begin{align}
\varphi^{r+1}_{k} = \Bigg [ \varphi^r_{k}+ & \gamma_k  \Bigg ( \frac{\partial h_k(\varphi, \xi)}{\partial \varphi} \Bigg |_{   \varphi =\varphi^r_{k}, ~  \xi =\xi_{k}^r } \Bigg)^\top     \nonumber \\
& \times
\Big( v^*  - h_k\big(\varphi^r_{k},\xi_{k}^r \big)  \Big)  \Bigg ]_{\underline{ q}_{k}}^{\overline{ q}_{k}}, \quad r \geq 0, \label{PSGD_alg}
\end{align}
where $\gamma_k$ is a positive constant,  and $\xi_k^r,~r \geq 0,$ denote  samples  from the distribution of $\Xi_k$. The algorithm in \eqref{PSGD_alg} is essentially a PSGD  algorithm  for the problem in \eqref{model_based_opt} (see, e.g., \cite{kushner2003stochastic} and the references therein).

  We initialize the PSGD algorithm in \eqref{PSGD_alg}  as follows
\begin{align}
\varphi^0_{k} = q_{k-1}, \nonumber \\
\xi_k^0=\xi_{k-1},
\end{align}
where $\xi_{k-1}$ denotes the realized value of $\Xi_{k-1}$, 
and instead  of running it to completion, we will only execute one iteration and set $q_{k}=\varphi^1_{k}$. Then, it follows that the evolution of $q_{k}$ is governed by
\begin{align}
q_{k}=  \Bigg [q_{k-1}+  & \gamma_k  \Bigg ( \frac{\partial h_k(\varphi, \xi)}{\partial \varphi} \Bigg |_{ \varphi=q_{k-1}, ~ \xi=\xi_{k-1}  } \Bigg)^\top \nonumber \\
& \times \Big( v^*  - h_k\big(q_{k-1},\xi_{k-1} \big)  \Big)  \Bigg ]_{\underline{ q}_{k}}^{\overline{ q}_{k}}.  \label{one_step_PSGD_alg} \end{align}
The one-step PSGD-based algorithm proposed  in \eqref{one_step_PSGD_alg} for solving the problem in \eqref{model_based_opt} is a special case of the setting in \cite{Nedich_2019}, which proposes a framework to sequentially solve stochastic  minimization problems with slowly varying cost functions that are convex. 
\begin{remark}
{We assume for subsequent developments that the problem \eqref{model_based_opt} changes slowly with $k$, while the one-step PSGD-based algorithm is executed, which entails assuming that the $h_k(\cdot,\cdot)$'s, $\underline{ q}_{k}$ and $\overline{ q}_{k}$ vary slowly with $k$. However, the algorithm will also work in the presence of large but infrequent variations due to, e.g., weather changes, as long as there are sufficiently long time intervals during which the algorithm is able to converge.}
\end{remark}

Note that by using $\xi_{k-1}$ to initialize the PSGD-based algorithm,  we are using a sample from $\Xi_{k-1}$;  however, recall that we have assumed  that the $\Xi_l$'s are i.i.d., so  we are effectively   sampling from  the distribution of $\Xi_k$. Then, since we have assumed that $h_k(\cdot,\cdot)$ changes slowly with $k$, we can make the following approximations:
\begin{align}
    h_k\big(q_{k-1},\xi_{k-1}\big) &\approx  h_{k-1}\big(q_{k-1},\xi_{k-1}\big) \nonumber \\
    & = v_{k-1}, \label{approx_1}  \\ 
 \frac{\partial h_k(\varphi, \xi)}{\partial \varphi} \Bigg |_{ \varphi=q_{k-1}, ~ \xi=\xi_{k-1}}  & \approx \frac{\partial h_{k-1}(\varphi, \xi)}{\partial \varphi} \Bigg |_{ \varphi=q_{k-1},~   \xi=\xi_{k-1}}. \label{approx_2}
\end{align}
Thus,  
if we assume measurements of  $v_{k},~k \geq 1,$ are available,  we can update the value of $q_{k}$ as follows:
\begin{align}
q_{k}=  \Big [q_{k-1}+  \gamma_k   S_{k-1} ^\top   \big( v^*  - v_{k-1} \big)   \Big ]_{\underline{ q}_{k}}^{\overline{ q}_{k}}, \label{model_based_controller}
\end{align}
where
\begin{align}
 S_{k-1}= \frac{\partial h_{k-1}(\varphi, \xi)}{\partial \varphi} \Bigg |_{ \varphi=q_{k-1}, ~ \xi=\xi_{k-1}} \in \mathbb{R}^{n \times m}. \label{sensitivity_matrix_k_1}
 \end{align}
 
Note that in order to execute \eqref{model_based_controller}, we would need to compute $\partial h_{k-1}(\varphi,\xi)/\partial \varphi$ for $\varphi=q_{k-1}$ and $\xi=\xi_{k-1}$. This computation can be done by manipulating the power flow Jacobian without necessarily solving the power flow equations; however, it requires knowing the value of $\xi_{k-1}$, which we have assumed  it is not a priori known.  To circumvent this issue, instead of using $S_{k-1}$ in \eqref{model_based_controller}, we will use an estimate obtained using measurements of   $\big \{q_{l}, v_{l} \big \}_{l=0}^{k-1}$; the construction of such an estimate is detailed next.

\section{Online Matrix Sensitivity Estimator}
\label{sec:onlineEstimator}
Define $\Delta v_{k} : = v_{k}-v_{k-1}$,  $\Delta q_{k} : = q_{k}-q_{k-1}$ and $\Delta \xi_{k} : = \xi_{k}-\xi_{k-1}$; then, by using   \eqref{input_output_model_2}, we have that
\begin{align}
v_{k-1}+\Delta v_{k} = h_k\big(q_{k-1}+\Delta q_{k}, \xi_{k-1} +\Delta \xi_{k} \big), \label{input_outoput_model_2}
\end{align}
Now,  by expanding the right-hand side of \eqref{input_outoput_model_2} about $\big( q_{k-1}, \xi_{k-1} \big)$ using the Taylor series expansion, and using the approximation in \eqref{approx_2}, 
it follows that
\begin{align}\label{model}
\Delta v_{k} = S_{k-1}\Delta q_{k} + \epsilon_{k},
\end{align}
where $\epsilon_{k}$ represents higher order terms in $\Delta q_{k}$,  all the terms in $\Delta \xi_{k}$, and the error associated with  the use of \eqref{approx_2}. Then, by assuming that  $\epsilon_{k}$ is much smaller that $S_{k-1}\Delta q_{k} $, we have that 
\begin{align}
\Delta v_{k} \approx S_{k-1}\Delta q_{k}.
\end{align}
 
Now let $\widehat{S}^*_{k}$ denote an estimate of $S_{k}$, which  we can obtain  using the method of   least squares as follows:
\begin{align}\label{ls_problem}
\widehat{S}_{k}^* = \underset{S}{\argmin} ~\ell(S),
\end{align}
where
\begin{equation*}
    \ell(S) := \sum_{l=1}^{k}\lambda^{k-l}\big\|\Delta v_{l} - S\Delta q_{l}\big\|_2^2,
\end{equation*}
with $\lambda \in (0,1)$ denoting the forgetting factor that allows the estimator to assign exponentially less weight to older measurements and adapt to changes in operating conditions. The solution to the least-squares problem \eqref{ls_problem} can be derived as follows. We first compute the gradient of the loss function $\ell(S)$ in \eqref{ls_problem}, which yields the following expression:
\begin{align}\label{ls_grad}
    \nabla \ell(S) = -2\sum_{l=1}^k\lambda^{k-l}\big(\Delta v_{l} - S\Delta q_{l}\big)\Delta q_{l}^\top.
\end{align}
Then, we equate the gradient \eqref{ls_grad} to zero and obtain
\begin{align}\label{ls_grad2}
    \sum_{l=1}^k\lambda^{k-l}\Delta v_{l}\Delta q_{l}^\top = S\left(\sum_{l=1}^k\lambda^{k-l}\Delta q_{l}\Delta q_{l}^\top\right).
\end{align}
Now in order to solve for $S$ in \eqref{ls_grad2}, we need to invert the matrix that multiplies  $S$. Since this matrix is the sum of  $k$ $(m \times m)$-dimensional rank-one matrices, sufficient conditions for ensuring its invertibility are that: i) $k \geq m$,  and ii) the sequence $\{\Delta q_{l} \}_{l=1}^k$ is persistently exciting.\footnote{A discrete-time signal $x[t]$, $t=1,2,\dots,$ is persistently exciting if, for every $k$, there exist an integer $l$ and constants $\varrho_1,\varrho_2>0$ such that the matrices
$\varrho_1I-\sum_{t=k}^{k+l} x[t]x[t]^\top$ and $\sum_{t=k}^{k+l} x[t]x[t]^\top - \varrho_2I$ are positive definite \cite{AW:08}.}  Assuming these two conditions are satisfied, we can now solve for $S$ in \eqref{ls_grad2}  and obtain $\widehat{S}_{k}^*,~k \geq m$: 
\begin{align}\label{ls_sol}
    \widehat{S}_{k}^*&=\left(\sum_{l=1}^k\lambda^{k-l}\Delta v_{l}\Delta q_{l}^\top\right) \Bigg(\sum_{l=1}^k\lambda^{k-l}\Delta q_{l}\Delta q_{l}^\top\Bigg)^{-1}.
\end{align}

The issue of ensuring that the sequence $\{\Delta q_{l} \}_{l=1}^k$ is persistently exciting is addressed in the next section. However, even if this can be addressed satisfactorily, the estimate $\widehat{S}_{k}^*$ obtained using \eqref{ls_sol} is only valid for $k \geq m$; this means that in practice we would have to  wait for $k$ steps before we can obtain our first estimate. 
To address this, we use an algorithm that will recursively generate a sequence $\big \{\widehat{S}_k \big \}_{k \geq 1}$ that can be shown to converge  to $\widehat{S}_{k}^*$ for $k$ large enough. The update rules for this recursive algorithm are as follows:
\begin{align}
\widehat{S}_{k}  =&~\widehat{S}_{k-1} + \Big (\Delta v_{k} - \widehat{S}_{k-1}  \Delta q_{k}    \Big) \Delta q_{k}^\top F_k, \nonumber \\
F_{k}   =&~ \lambda^{-1}  F_{k-1} \nonumber \\
&-\frac{\lambda^{-2}}{1+ \lambda^{-1} \Delta q_{k} ^\top F_{k-1} \Delta q _{k}}   F_{k-1} \Delta q _k \Delta q^\top _k  F_{k-1}, \label{sensitivity_update_eqns}
\end{align}
with
\begin{align}
  \widehat{S}_0 &= \frac{\partial h (\varphi, w)}{\partial \varphi} \Bigg |_{ \varphi=q_{0}, ~  w=w_0^\circ}, \nonumber \\
  F_0&=I_m, \label{init}
\end{align}
where $I_m$ denotes the $(m \times m)$-dimensional identity matrix. 
The algorithm is obtained by using the matrix inversion lemma to recursively invert the matrix $\sum_{l=1}^k\lambda^{k-l}\Delta q_{l}\Delta q_{l}^\top$; see Appendix~\ref{rLSE_derivaiton} for the derivation. 
\begin{remark}
Note that in order to initialize the algorithm in \eqref{sensitivity_update_eqns}, we need to compute $\partial h(\varphi,w)/\partial \varphi$ for $\varphi=q_0$ and $w=w_0^\circ$. Unlikely the computation of $S_{k-1}$ in \eqref{sensitivity_matrix_k_1}, in this case, we can indeed compute $\widehat{S}_0$ by manipulating the power flow Jacobian without necessarily solving the power flow equations because   $q_0$ and $w_0^\circ$ are known. Furthermore, the initialization of $\widehat{S}_0$ as in \eqref{init} is not crucial as we will see in the numerical simulation results presented in Section~\ref{sec:numericalResults}.
\end{remark}
\begin{remark}
Note that if the algorithm in  \eqref{sensitivity_update_eqns} were to be executed for $k > m$, with $\widehat{S}_m = \widehat{S}_m^*$  obtained using \eqref{ls_sol} for $k=m$, and 
$$F_{m}=\Bigg(\sum_{l=1}^m\lambda^{k-l}\Delta q_{l}\Delta q_{l}^\top\Bigg)^{-1},$$
the sequence generated by the algorithm would be $\big \{\widehat{S}_l^* \big \}_{l \geq m} $, i.e., for each $k >m$, the algorithm would generate the exact solution to \eqref{ls_grad2} as given in \eqref{ls_sol}.
\end{remark}

\section{Online Feedback Optimization}
\label{sec:feedbackOptimization}
Now, on the one hand, by replacing $S_{k-1}$ by $\widehat{S}_{k-1}$ in \eqref{model_based_controller}, we obtain that 
\begin{align}
q_{k}=  \Big [q_{k-1}+  \gamma_k   \widehat{S}_{k-1}    ^\top \big( v^*  - v_{k-1} \big)    \Big ]_{\underline{ q}_{k}}^{\overline{ q}_{k}}, \quad k \geq 1. \label{online_based_controller}
\end{align}
On the other hand, by inspecting \eqref{ls_sol}, one can see that  an estimate of  $\widehat{S}_{k}$ can be computed if the matrix $$ \sum_{l=1}^k\lambda^{k-l}\Delta q_{l}\Delta q_{l}^\top$$ is invertible; this can be ensured as mentioned earlier if the sequence $\big \{\Delta q_{l} \big \}_{l=1}^{k}$ is persistently exciting. However, it is not clear a priori that   \eqref{online_based_controller} will generate such persistently exciting sequence. To address this issue, we modify \eqref{online_based_controller} to add a mechanism to ensure that indeed the sequence $\big \{\Delta q_{l} \big \}_{l=1}^{k}$ is persistently exciting. First, by assuming that $q_{k-1} \in [\underline{q}_{k-1},\overline{q}_{k-1} ]$, we can rewrite \eqref{online_based_controller} as follows:
\begin{align} \label{modified_online_based_controller_1}
q_{k}=  q_{k-1}+  \Big [ \gamma_k   \widehat{S}_{k-1}    ^\top  \big (v^*  - v_{k-1} \big )    \Big ]_{\underline{  \Delta  q}_{k}}^{\overline{ \Delta   q}_{k}}, \quad k \geq 1, 
\end{align}
where $\underline{ \Delta  q}_{k}=  \underline{q}_{k} - q_{k-1}$ and $\overline{ \Delta  q}_{k}=  \overline{q}_{k} - q_{k-1}$. Then, following the ideas in \cite{MadiADG2021}, we will make three modifications to \eqref{modified_online_based_controller_1} as follows:
\begin{enumerate}
    \item[\textbf{M1.}]
The term inside the projection term on the right hand side of \eqref{modified_online_based_controller_1} is modified as follows:
\begin{align}
 \gamma_k   \widehat{S}_{k-1}    ^\top \big( \Delta  v^*_{k}   + cz_k    \big),
\end{align}
where $c$ is some constant, $\Delta v^*_{k}=v^*  - v_{k-1}$,   and
\begin{equation}\label{w_def}
z_k=
\begin{cases}
0,~\text{if}~\{\Delta v^*_{l}\}_{l=1}^k~\text{is persistently exciting},\\
\text{sampled from}~(-\Delta v^*_{k})U(0,a_1),~\text{otherwise},
\end{cases}
\end{equation}
with $a_1\in(0,1)$ and  $U(x,y)$ denoting the continuous uniform distribution over the interval $(x,y)$. 
    \item[\textbf{M2.}]The incremental lower and upper capacity limit in \eqref{modified_online_based_controller_1} are respectively set to
\begin{align}
\underline{q}_{k}+\underline{\eta}_{k}, \nonumber \\
\overline{q}_{k}-\overline{\eta}_{k}, \label{incremental_limits_modified}
\end{align}
where  $\underline{\eta}_{k}$ is sampled from $U(0,a_2|\underline{\Delta q}_{k}|)$ and $\overline{\eta}_{k}$ is sampled from $U(0,a_2|\overline{\Delta q}_{k}|)$, with $a_2\in(0,1)$. 
    \item[\textbf{M3.}]
By using the two modifications above, 
define the following quantity:
 \begin{align}
 \widetilde{ \Delta q}_{k}=  \Big [ \gamma_k   \widehat{S}_{k-1}    ^\top  \big( \Delta  v^*_{k}   + cz_k    \big)    \Big ]_{\underline{  \Delta  q}_{k} +\underline{\eta}_{k} }^{\overline{ \Delta   q}_{k} -\overline{\eta}_{k}}. \label{updated_limits}
 \end{align}
Let $\mbox{Null}\big(\widehat{S}_{k}\big)$ denote the null space of the matrix $\widehat{S}_{k}$. 
Then, if the sequence $\Big \{ \big \{\Delta q_{l} \big \}_{l=1}^{k-1},  \widetilde{ \Delta q}_{k}  \Big \}$ is persistently exciting, we update the value of $q_{k}$ as follows:
\begin{align}
q_{k}=q_{k-1}+ \widetilde{ \Delta q}_{k}, \quad k \geq 1,  \label{q_update_final_1}
\end{align}
otherwise we update its value as follows:
\begin{align}
q_{k}=q_{k-1}+ \widetilde{ \Delta q}_{k} +  \gamma_k \alpha_{k} \nu_{k},  \quad k \geq 1, \label{q_update_final_2}
\end{align}
with $\alpha_{k}$   sampled from $U(-b_k,b_k)$, where $b_k\geq 0$ is arbitrarily chosen so that
\begin{align}
 \underline{ \Delta  q}_{k}  &\leq \widetilde{ \Delta q}_{k} + \gamma_k b_{k}\nu_{k} \leq  \overline{ \Delta  q}_{k}, \label{sens_null_space_perturb1_aux}\\
 \underline{ \Delta  q}_{k}  &\leq \widetilde{ \Delta q}_{k} - \gamma_k \alpha_{k}\nu_{k} \leq  \overline{ \Delta  q}_{k}, 
 \label{sens_null_space_perturb2_aux}
\end{align}
for some   arbitrarily chosen $\nu_k    \in\mbox{Null}\big(\widehat{S}_{k}\big)$; 
this ensures that $\underline{ q}_{k} \leq q_{k} \leq  \overline{ q}_{k}$.
\end{enumerate}

The  rationale behind the modifications above is to introduce excitation across the  entire space in which the $\Delta q_k$'s take values. To see this, first  consider the case when the capacity constraints are not active at instant $k$, i.e.,  $
 \widetilde{ \Delta q}_{k}=   \gamma_k   \widehat{S}_{k-1}    ^\top  \big( \Delta  v^*_{k}   + cz_k    \big).$  Let $\mbox{Row}\big(\widehat{S}_{k}\big)$ denote the row space of  $\widehat{S}_{k} \in \mathbb{R}^{n \times m}$. Then, the idea behind adding the term $cz_k$ to $\Delta v_k^*$ in Modification~M1 is to introduce excitation in the subspace spanned by the columns of $\widehat{S}^\top_{k},$ i.e., $\mbox{Row}\big(\widehat{S}_{k}\big)$, 
 whereas the idea behind adding the term $\gamma_k b_{k}\nu_{k}$ in Modification~M3 is to introduce excitation in $\mbox{Null}\big(\widehat{S}_{k}\big)$, which is the orthogonal complement of $\mbox{Row}\big(\widehat{S}_{k}\big)$. Then, since for any $x \in \mathbb{R}^m$ we have that $x=u+v,$ with $u \in \mbox{Row}\big(\widehat{S}_{k}\big)$ and  $v \in \mbox{Null}\big(\widehat{S}_{k}\big)$, this mechanism allows us to introduce excitation in separate parts of the whole space $\mathbb{R}^m$ as needed. The idea behind modifying the incremental capacity limits as described in Modification~M2 is as follows. If the capacity constraints are active at instant $k$, by modifying the incremental capacity limits as in \eqref{incremental_limits_modified}, we are leaving headroom to add the term $\gamma_k b_{k}\nu_{k}$ if needed,  while still ensuring that the constraints in \eqref{sens_null_space_perturb1_aux} -- \eqref{sens_null_space_perturb2_aux} imposed by the actual incremental capacity limits are satisfied.


\section{Convergence Analysis} \label{sec:stabilityAnalysis}
Next, we provide  the conditions under which the 
proposed online feedback optimization controller in \eqref{updated_limits} -- \eqref{sens_null_space_perturb2_aux} converges. To this end, we first need to rewrite the model in \eqref{input_output_model} as follows.   Define $\delta v_k=v_k-v^*$, $\delta q_k=q_k-q_0$, and   $\delta w_k=w_k-w_0^\circ$ (note that $w_k-w_0^\circ=w_k^\circ-w_0^\circ+\xi_k$); then,   we have that
 \begin{align}
 v^*+\delta v_k&=h(q_0 +\delta q_k,w_0^\circ+\delta w_k), \quad k \geq 0. \label{input_output_model_5}
 \end{align}
Now by  using the Taylor series expansion  to expand the right hand side of \eqref{input_output_model_5} about $(q_0,w_0^\circ)$, we obtain that
 \begin{align}
     v_k= S_\varphi q_k+S_w \xi_k+\mu_k+\eta_k,\label{power_system}
 \end{align}
 where 
  \begin{align}
     S_\varphi&=\frac{\partial h (\varphi,w)}{\partial \varphi} \Bigg |_{ \varphi=q_{0}, ~ w=w_0^\circ}  \nonumber \\
     S_w&=\frac{\partial h (\varphi,w)}{\partial w} \Bigg |_{ \varphi=q_{0}, ~ w=w_0^\circ}  \nonumber \\
 \mu_k&=v^*-S_\varphi q_0+S_w \big (w_k^\circ-w_0^\circ\big), 
 \end{align}
with $\eta_k$ representing higher-order terms in $\delta q_k$ and   $\delta w_k$.

The next result, which is established using standard analysis techniques (see, e.g., \cite{doi:10.1137/070704277}),  shows that the sequence $\{q_k\}_{k \geq 1}$ generated by \eqref{w_def} -- \eqref{sens_null_space_perturb2_aux} converges almost surely to a solution of the optimization problem in \eqref{model_based_opt} provided that the sensitivity estimates, $\big \{\widehat{S}_{k} \big \}_{k \geq 1}$, are unbiased. In doing so, we set $c=0$ in \eqref{updated_limits}. Also, we  assume that (i) $\mu_k=\mu,~k \geq 0,$ where $\mu$ is some constant, which is consistent with the assumption we made earlier that $w_k^\circ$ slowly changes with $k$; and (ii) $\eta_k=0,~k \geq 0$, which is reasonable since this term captures the higher-order terms of the Taylor series expansion of $h(q,w)$ about $(q_0,w_0^\circ)$. We plan to address these   issues in future work.

\begin{prop}\label{prop:convergence_results} 
Consider the online feedback optimization controller in \eqref{w_def} -- \eqref{sens_null_space_perturb2_aux} with  $c=0$,  and  where the sequence $\big \{\widehat{S}_{k} \big \}_{k \geq 0}$ are the unbiased estimates of the sensitivity   matrix $S_\varphi$, namely, $\text{E}\Big [\widehat{S}_{k}-S_\varphi \,| \, \mathcal{F}_{k}\Big ] = 0$, with $\mathcal{F}_{k}$ denoting the accumulated collection of states up to instant $k$, $\big \{(v_l,q_l) \big \}_{l=0}^{k}$, and $\text{E}\left[\big \|\widehat{S}_{k}-S_\varphi \big \|^2\right]\leq \sigma_\varphi$, with $\sigma_\varphi$ denoting some positive constant.
Let $\mathcal{X}_k^*$ denote the set    of optimal solutions of \eqref{model_based_opt}. We assume that the following conditions hold:

{(a) $\underline{q}_k$ and  $\overline{q}_k$ are constant, $\mu_k$ is constant, $\text{E}\big[\eta_k\big]=0$ and $\text{E}\big[\|\eta_k\|^2_2\big]\leq \sigma$.}

(b) $\gamma_k$ is a diminishing step size, namely, 
\begin{equation}
    \sum_{k=1}^{\infty} \gamma_k = \infty, \quad \sum_{k=1}^{\infty} \gamma_k^2 < \infty.
\end{equation}
Then,  the sequence $\{q_k\}_{k \geq 1}$ converges almost surely to some point in $\mathcal{X}_k^*$.
\end{prop}

\begin{proof}
Define \begin{align*}
    f(\varphi) &:= \frac{1}{2} \text{E}_{\Xi_k} \left[\big \| v^*  - h_k\big(\varphi,\Xi_{k}\big) \big \|_2^2\right],\\
    g_{k} &:= \frac{\partial f(\varphi)}{\partial \varphi} \Bigg |_{ \varphi=q_{k}} = -S_{\varphi}^\top(v^*-S_{\varphi}q_{k}-\mu_{k}),
\end{align*}
$\delta_{k} := (S_{\varphi}-\widehat{S}_{k})^\top \Delta  v^*_{k+1}$; then,  it follows from \eqref{updated_limits} -- \eqref{q_update_final_2} that
\begin{align}
q_{k+1} &= \Big[q_{k} + \gamma_{k+1}\widehat{S}_{k}^\top\Delta  v^*_{k+1}+\gamma_{k+1}\alpha_{k+1}\nu_{k+1}\Big]_{\underline{q}_{k+1} +\underline{\eta}_{k+1} }^{\overline{ q}_{k+1} -\overline{\eta}_{k+1}}\nonumber\\
    &=\Big[q_{k} + \gamma_{k+1}\big(S_{\varphi}+(\widehat{S}_{k}-S_{\varphi})\big)^\top\Delta  v^*_{k+1}\nonumber\\&\quad+\gamma_{k+1}\alpha_{k+1}\nu_{k+1}\Big]_{\underline{q}_{k+1} +\underline{\eta}_{k+1}  }^{\overline{ q}_{k+1} -\overline{\eta}_{k+1}}\nonumber\\
    &= \Big[q_{k} -  \gamma_{k+1}({g_{k}}+S_{\varphi}^\top S_w{\xi_{k}}+S_{\varphi}^\top \eta_{k} \nonumber \\ &\quad+\delta_{k}-\alpha_{k+1}\nu_{k+1})\Big]_{\underline{q}_{k+1} +\underline{\eta}_{k+1} }^{\overline{ q}_{k+1} -\overline{\eta}_{k+1}}.
\end{align} 
Define $\zeta_k := S_{\varphi}^\top S_w\xi_{k}+S_{\varphi}^\top\eta_k +\delta_{k}-\alpha_{k+1}\nu_{k+1}$; then, by   using the non-expansiveness property of the Euclidean projection operator, we have that
\begin{align}
    \|q_{k+1}-q^*\|^2&\leq \big\|q_{k}-q^* - \gamma_{k+1}\big(g_{k}+\zeta_{k}\big)\big\|^2\nonumber\\
    &= \|q_{k}-q^*\|^2 - 2\gamma_{k+1}\big(g_{k}+\zeta_{k}\big)^\top (q_{k} - q^*)\nonumber\\&\quad + \gamma_{k+1}^2\|g_{k}+\zeta_{k}\|^2.
    \label{main_ineq1}
\end{align}
It follows from the convexity property that
\begin{align}
    -g_{k}^\top (q_{k}-q^*) \leq f(q^*) - f(q_{k})\label{grad_ineq}.
\end{align}
By applying \eqref{grad_ineq} to \eqref{main_ineq1}, we obtain that
\begin{align}
    \|q_{k+1}-q^*\|^2 
    &\leq\|q_{k}-q^*\|^2 -2\gamma_{k+1}(f(q_{k}) - f(q^*))\nonumber\\&\quad-2\gamma_{k+1}\zeta_{k}^\top (q_{k} - q^*) + \gamma_{k+1}^2\|g_{k}+\zeta_{k}\|^2.
    \label{main_ineq2}
\end{align}
Define \[g^* := \frac{\partial f(\varphi)}{\partial \varphi} \Bigg |_{ \varphi=q^*};\]
then, the following Lipschitz condition holds trivially for some $L>0$:
\begin{align}
    \|g_{k}-g^*\| = \|S_{\varphi}^\top S_{\varphi}(q_{k}-q^*)\| \leq L\|(q_{k}-q^*)\|.
    \label{Lipschitz}
\end{align}
By using \eqref{Lipschitz} and the fact that $2x^\top y\leq x^2+y^2$, for any $x,y\in\mathds{R}^n$, we have that
\begin{align}
    \|g_{k}+\zeta_{k}\|^2 &= \big\|g_{k}-g^*+g^*+\zeta_{k}\big\|^2\nonumber\\
    &\leq 2\|g_{k}-g^*\|^2+2\|g^*+\zeta_{k}\|^2\nonumber\\
    &\leq 2L^2\|q_{k}-q^*\|^2+2\|g^*+\zeta_{k}\|^2.
    \label{lip_ineq}
\end{align}
Since $\text{E}\big[\delta_{k}\,|\,\mathcal{F}_{k}\big] = 0$, $\text{E}\big [S_{\varphi}^\top S_w\xi_{k} \,| \, \mathcal{F}_{k} \big ] = 0$, $\text{E}\big [S_{\varphi}^\top \eta_{k} \,| \, \mathcal{F}_{k} \big ] = 0$, and $\text{E}\big [\alpha_{k+1}\nu_{k+1}\,|\,\mathcal{F}_k\big ] = 0$, we have that
\begin{align}
    \text{E}\big [\zeta_{k}\,|\,\mathcal{F}_{k}\big ] &= \text{E}\big [S_{\varphi}^\top S_w\xi_{k}+S_{\varphi}^\top\eta_k +\delta_{k}-\alpha_{k+1}\nu_{k+1}\,|\,\mathcal{F}_{k}\big ] \nonumber\\&= 0.\label{zeta_ineq}
\end{align}
Then, by taking an expectation of \eqref{main_ineq2} and applying \eqref{lip_ineq} and \eqref{zeta_ineq}, we obtain that
\begin{align}
    \text{E}\left[\|q_{k+1}-q^*\|^2\,|\,\mathcal{F}_{k}\right]
    &\leq(1+2L^2\gamma_{k+1}^2)\|q_{k}-q^*\|^2\nonumber\\
    &\quad -2\gamma_{k+1}(f(q_{k}) - f(q^*))\nonumber\\
    &\quad -2\gamma_{k+1} \text{E}[\zeta_{k}\,|\,\mathcal{F}_{k}]^\top (q_{k} - q^*)\nonumber\\
    &\quad+2\gamma_{k+1}^2 \text{E}\left[\|g^*+\zeta_{k}\|^2\,|\,\mathcal{F}_{k}\right])\nonumber\\
    &=(1+2L^2\gamma_{k+1}^2)\|q_{k}-q^*\|^2\nonumber\\
    &\quad  -2\gamma_{k+1}(f(q_{k}) - f(q^*))\nonumber\\&\quad+2\gamma_{k+1}^2 \text{E}\left[\|g^*+\zeta_{k}\|^2\,|\,\mathcal{F}_{k}\right]).
    \label{main_ineq3}
\end{align}
Further, it can be easily shown that
\begin{align}\text{E}\left[\|g^*+\zeta_{k}\|^2 \, |\, \mathcal{F}_{k}\right]<\infty. \label{noise_bound}
\end{align}
By applying the Robbins-Siegmund Theorem (see, e.g., \cite[Lemma~11]{Polyak}) to \eqref{main_ineq3}, we conclude that $q_k$ converges almost surely to some point in $\mathcal{X}_k^*$.
\end{proof}

The next result provides the cost error bound at the time instant when the sequence $\{\Delta q_k\}_{k \geq 1}$ is no longer persistently exciting. This indeed will be the case if the rLSE in \eqref{sensitivity_update_eqns} is used to generate the sequence $\{\widehat{S}_{k}\}_{k \geq 1}$ as $\widehat{S}_{k}$ will cease to be an unbiased estimate of the sensitivity matrix, $S_\varphi$, as $k \to \infty$.

\begin{prop}
Suppose we have for some $\alpha$ and $T>0$ that 
\begin{align}
\sum_{l=1}^T\lambda^{T-l}\Delta q_{l}\Delta q_{l}^\top < \alpha I.\label{excitation_loss}    
\end{align}
Suppose that $q_T>\underline{q}_{T}+\underline{\eta}_{T}$ and $q_T<\overline{q}_{T}-\overline{\eta}_{T}$.
Then, the following relation holds for any $q^*\in\mathcal{X}_T^*$:
\begin{align}
    f(q_T)-f(q^*) < (\overline{q}_T-\underline{q}_T)\frac{\sqrt{m\alpha}+B}{\gamma_T},\label{prop_result}
\end{align}
where $B$ is an upper bound for $\|\zeta_l\|$, $l\geq 1$.
\end{prop}
\begin{proof}
It follows from \eqref{excitation_loss} that $\alpha I - \sum_{l=1}^T\lambda^{T-l}\Delta q_{l}\Delta q_{l}^\top$ is a positive definite matrix. Hence, by \cite[Corollary~7.1.5]{Horn_Johnson}, its trace is positive, namely,
\begin{align}
    \Tr\left(\alpha I -  \sum_{l=1}^T\lambda^{T-l}\Delta q_{l}\Delta q_{l}^\top\right) > 0,\label{trace_ineq}
\end{align}
where $\Tr(A) = \sum_i a_{i,i}$ denotes the trace of matrix $A=[a_{i,j}]$.
Then, by using the assumption that $q_T>\underline{q}_{T}+\underline{\eta}_{T}$ and $q_T<\overline{q}_{T}-\overline{\eta}_{T}$, and applying the triangle inequality, we have that
\begin{align}
    \sqrt{m\alpha} &> \sqrt{\Tr\left(\sum_{l=1}^T\lambda^{T-l}\Delta q_{l}\Delta q_{l}^\top\right)}\nonumber\\ 
    &\geq \|\gamma_T(g_T+\zeta_T)\|\geq \gamma_T(\|g_T\|-\|\zeta_T\|)\nonumber\\
    &\geq\gamma_T\|g_T\|-B,
\end{align}
Hence,
\begin{align}
    \|g_T\|< \frac{\sqrt{m\alpha}+B}{\gamma_T}.\label{grad_ineq2}
\end{align}
By using \eqref{grad_ineq} and the fact  that $q_T \in [\underline{q}_T,\overline{q}_T]$  and $q^* \in [\underline{q}_T,\overline{q}_T]$, we obtain that
\begin{align}
    \|g_T\|(\overline{q}_T-\underline{q}_T)\geq f(q_T)-f(q^*).\label{grad_ineq3}
\end{align}
By combining \eqref{grad_ineq3} with \eqref{grad_ineq2} the result in   \eqref{prop_result} follows.
\end{proof}

 \begin{figure}[b]
 \vspace{-0.2in}
		\centering
		\includegraphics[width=0.57 \textwidth]{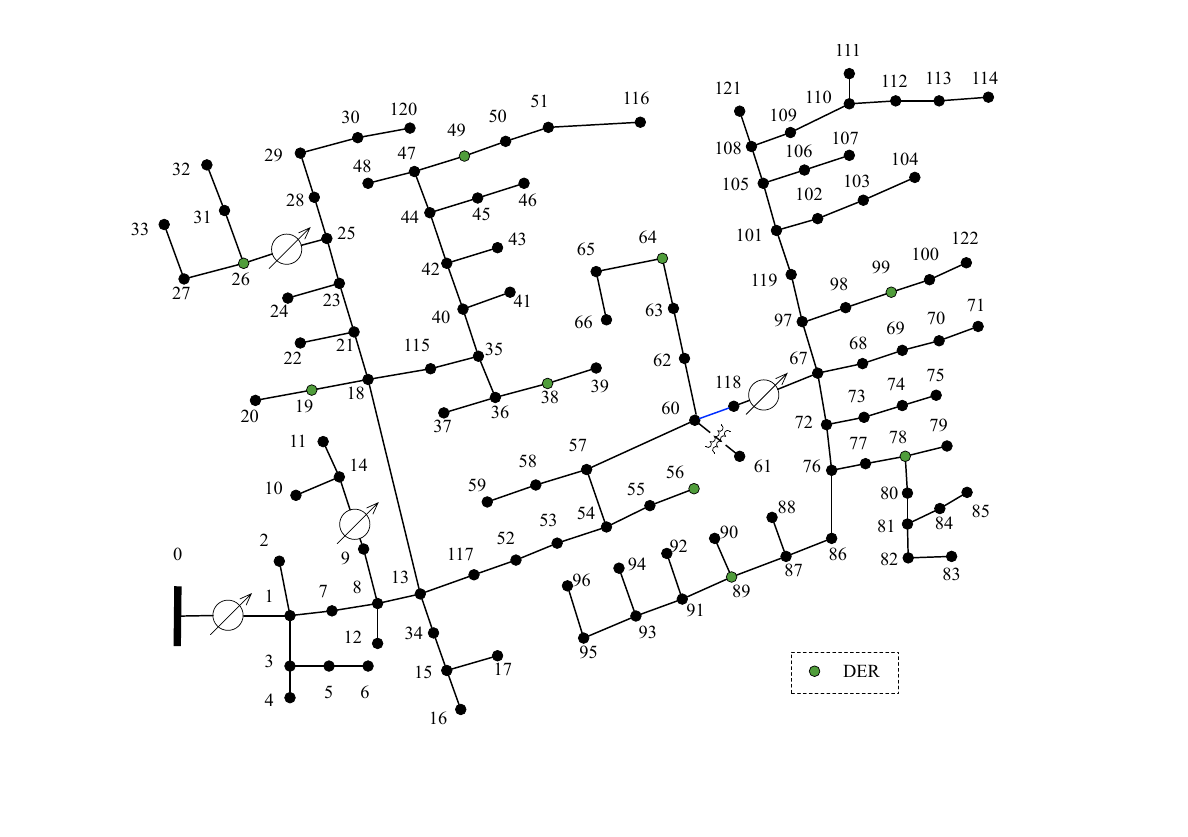}
		\vspace{-0.5in}
		\caption{IEEE 123-bus distribution test feeder.} \label{IEEE_123}
\end{figure}
\begin{figure}[h]
		\begin{center}
		\includegraphics[width=0.48\textwidth]{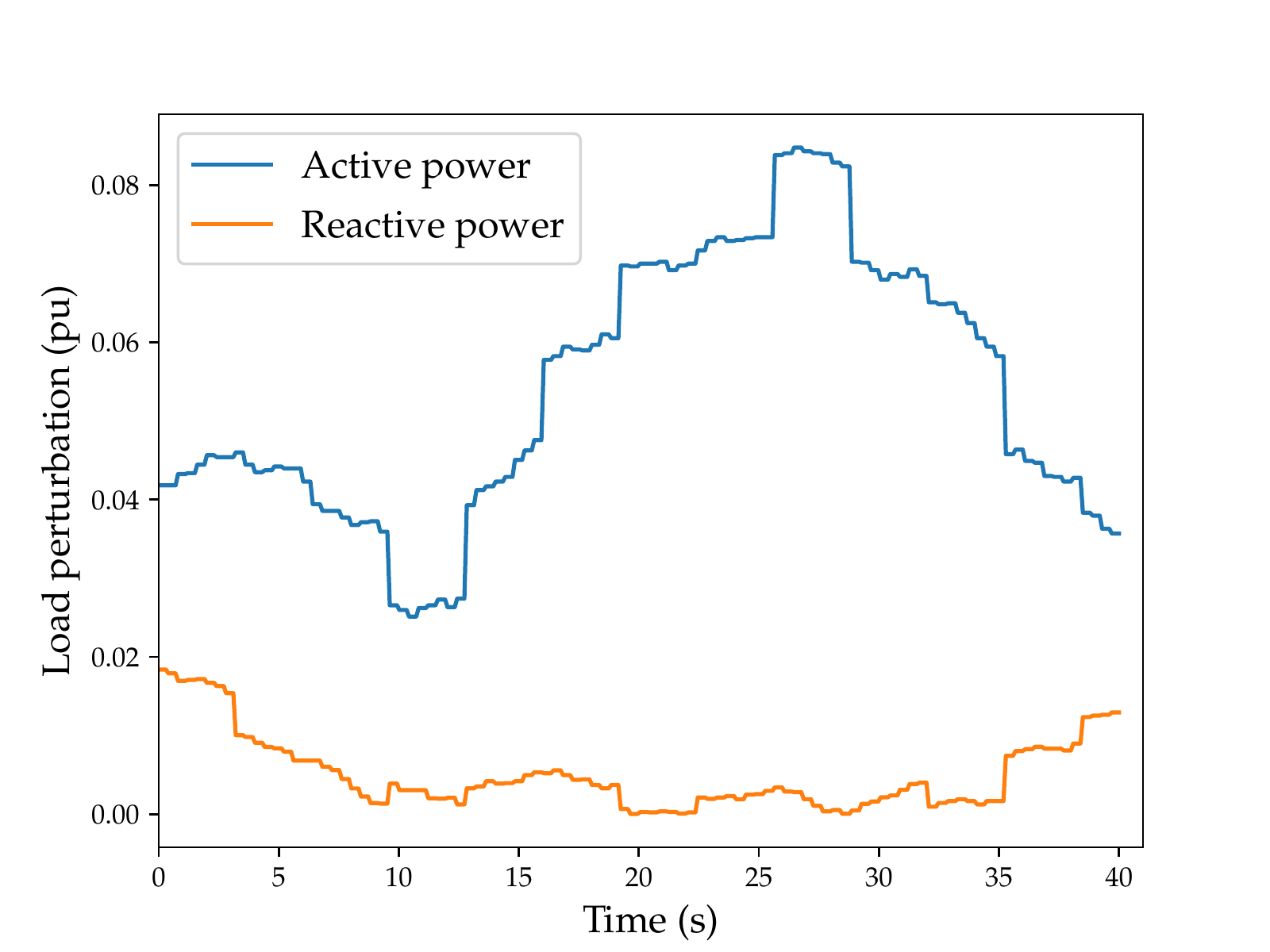}
		\vspace{-0.2in}
		\caption{Trajectories of active and reactive power at bus 1.} \label{load_perturbation}
		\end{center}
\end{figure}
 
 \section{Numerical Results}
 \label{sec:numericalResults}
Here,  we present numerical simulation results  illustrating the  effectiveness of our proposed online feedback optimization controller for voltage regulation in power distribution systems. To this end, we employ the test system  depicted in Fig.~\ref{IEEE_123}, which is a modified version of the three-phase balanced IEEE 123-bus distribution test feeder presented in \cite{XuDo19, ieee_pes_test_feeder}, with reactive-power-capable DERs added to the network at the following buses: $19$, $26$, $38$, $49$, $56$, $64$, $78$, $89$, and $99$. At other buses in the distribution network, random perturbations in the active and reactive power demand are introduced every $100$~milliseconds (for example, see Fig.~\ref{load_perturbation} for an illustration of the active and reactive power demand at bus~1). 
In this case we have that $n=122$ and $m=9$. Then, we set $F_0 = I_{9}$ and $\widehat{S}_{0}$ to the matrix that results from removing the last $113$  columns of the matrix $I_{122}$. Additionally, all the components of  $\underline{ q}_{k}$ and $\overline{ q}_{k}$ are set to $-0.5$~pu and $0.5$~pu, respectively, for all $k$. We also set $a_1 = 0.8$, $a_2 = 0.8$, $\lambda = 0.995$ and $\gamma_k = 0.95$, $k\in\{1,\dots,100\}$, $\gamma_k = 0.1$, $k\in\{101,\dots,400\}$.
~\\

\begin{figure}[h]
\vspace{-0.2in}
		\centering
		\hspace{0.49in}	\includegraphics[width=0.5\textwidth]{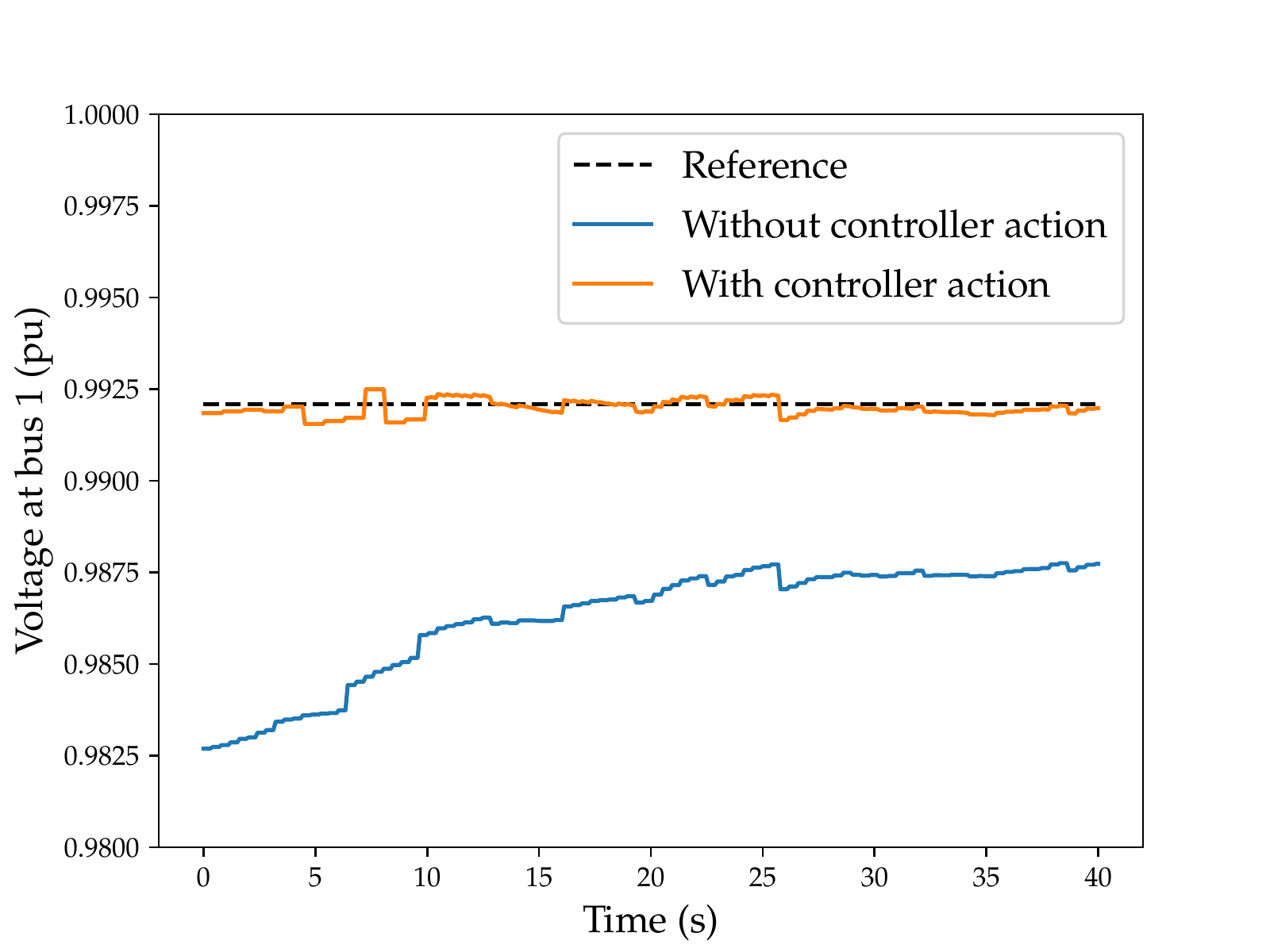}
		\vspace{-0.2in}	\caption{Trajectory of voltage magnitude at bus 1 for c = 0.} \label{Voltage trajectories c0}
\end{figure}

\begin{figure}[h]
\vspace{-0.02in}
		\centering
		\hspace{0.49in}	\includegraphics[width=0.485\textwidth]{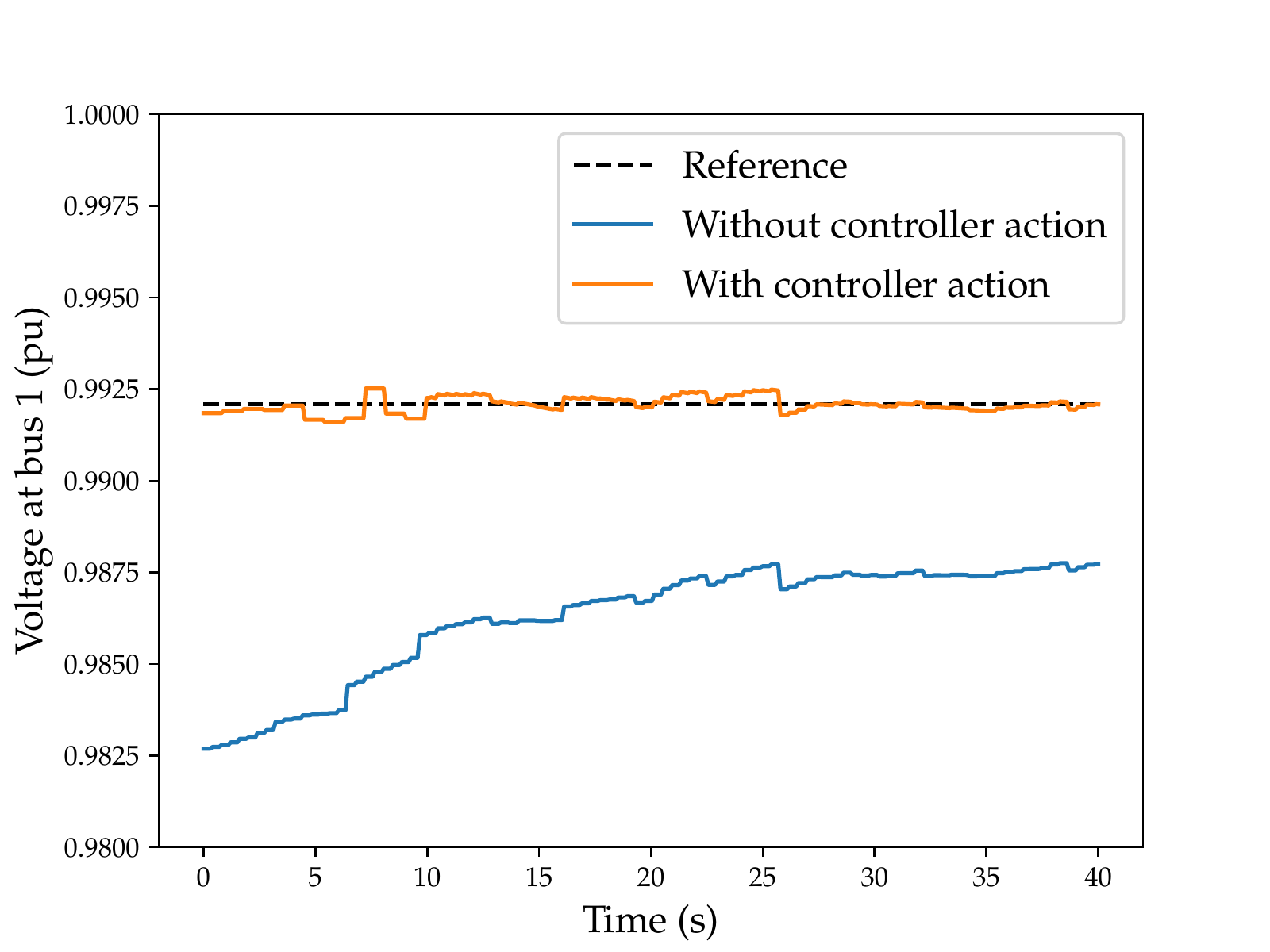}
		\vspace{-0.2in}	\caption{Trajectory of voltage magnitude at bus 1 for c = 0.5.} \label{Voltage trajectories c0.5}
\end{figure}

\begin{figure}[h]
\vspace{-0.1in}
		\centering
		\hspace{0.5in}	\includegraphics[width=0.5\textwidth]{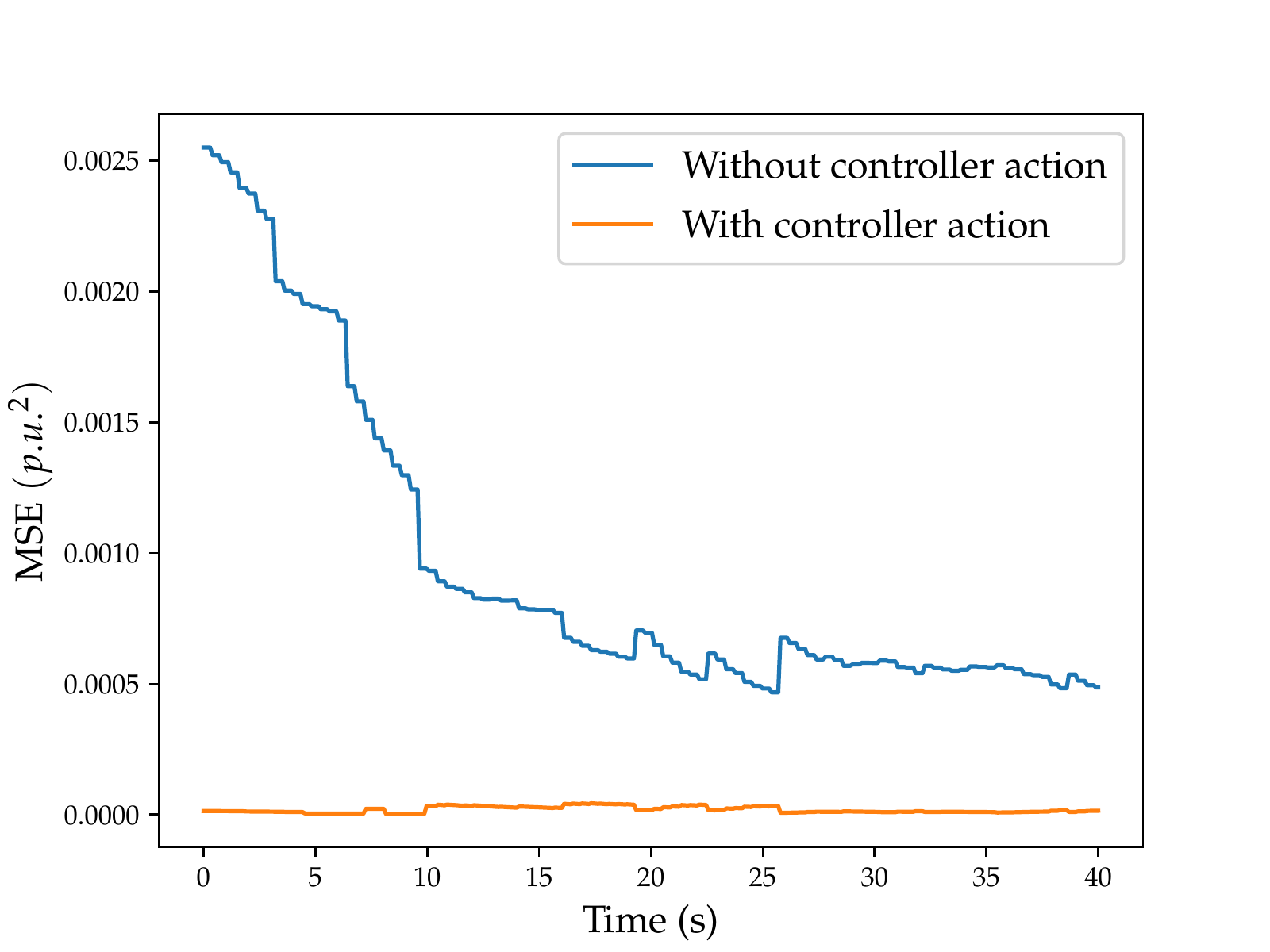}
		\vspace{-0.2in}	\caption{Trajectory of MSE of voltage magnitude for c = 0.} \label{MSE}
\end{figure}

Figures~\ref{Voltage trajectories c0} and ~\ref{Voltage trajectories c0.5} show the  trajectory followed by bus~1 voltage magnitude under i) no control action, i.e., $q_k=q_0$ for all $k \geq 1$, and ii) the action of the proposed controller for $c = 0$ and $c = 0.5$.  We also display the mean square error (MSE) of voltage magnitudes at all buses, i.e $\frac{\sum_{i=1}^n (v_i^*-v_i)^2}{n}$ in Fig.~\ref{MSE}. All our  numerical simulations demonstrate that our proposed controller is extremely effective at maintaining bus voltage magnitudes closed to their nominal value.

 \section{Concluding Remarks}
\label{sec:concludingRemarks}

In this paper we have proposed a controller for voltage regulation  in power distribution networks using reactive-power capable DERs. The proposed controller is based on a PSGD algorithm for solving online a sequence of optimization problems, each of which capturing the objectives and constraints of the voltage regulation problem at a particular time instant. By assuming the cost functions of each of these problems vary slowly with time, we can show that the PSGD-based algorithm acts as a feedback controller. In order to execute the controller, it is necessary to know the sensitivities of changes in bus voltage magnitudes with respect to changes in reactive power injections; we assume this are not a priori known and use a rLSE to estimate them. 

We showcased the performance of the controller via numerical simulations on the IEEE 123-bus system. In addition, under certain simplifying assumptions, we showed that the the sequence of DER setpoints generated by the controller converges almost surely to a solution of the aforementioned optimization problem when the estimates used by the controller are unbiased. In this regard, while the estimates generated by the rLSE will initially be unbiased, in  the limit this will no longer be the case; however, the simulation results show that the controller still performs satisfactorily. We plan to investigate this issue in future work and also investigate  the effect on the  convergence analysis of relaxing  the other assumptions we made in establishing the results in this paper. 

\appendix

\subsection*{Derivation of Sensitivity Estimator}\label{rLSE_derivaiton}
In order to derive \eqref{sensitivity_update_eqns}, consider the expression in \eqref{ls_sol}:
 \begin{align}\label{ls_sol_APP}
    \widehat{S}_{k}&=\left(\sum_{l=1}^k\lambda^{k-l}\Delta v_{l}\Delta q_{l}^\top\right) \Bigg(\sum_{l=1}^k\lambda^{k-l}\Delta q_{l}\Delta q_{l}^\top\Bigg)^{-1},
\end{align}
Define 
\begin{align}
   F_{k}^{-1}= \sum_{l=1}^k\lambda^{k-l}\Delta q_{l}\Delta q_{l}^\top;
\end{align}
then, we have that
\begin{align}
   F_{k}^{-1} &= \lambda \sum_{l=1}^{k-1}\lambda^{k-1-l}\Delta q_{l}\Delta q_{l}^\top + \Delta q_{k}\Delta q_{k}^\top \nonumber \\
   & = \lambda F_{k-1}^{-1} + \Delta q_{k}\Delta q_{k}^\top. \label{F_update_APP}
\end{align}
Then, by using the matrix inversion lemma (see, e.g.,   \cite{Horn_Johnson}), it follows that
\begin{align}
  F_{k}&=\Big (  \lambda F_{k-1}^{-1} + \Delta q_{k}\Delta q_{k}^\top\Big)^{-1}  \nonumber \\
  & =  \lambda^{-1}  F_{k-1} \nonumber \\
&-\frac{\lambda^{-2}}{1+ \lambda^{-1} \Delta q_{k} ^\top F_{k-1} \Delta q _{k}}   F_{k-1} \Delta q _k \Delta q^\top _k  F_{k-1}.
\end{align}
Now, we manipulate the expression in  \eqref{ls_sol_APP} as follows:
\begin{align}
        \widehat{S}_{k}&=\left(\sum_{l=1}^k\lambda^{k-l}\Delta v_{l}\Delta q_{l}^\top\right)  \Bigg(\sum_{l=1}^k\lambda^{k-l}\Delta q_{l}\Delta q_{l}^\top\Bigg)^{-1} \nonumber \\
        & = \left( \lambda \underbrace{\sum_{l=1}^{k-1}\lambda^{k-1-l}\Delta v_{l}\Delta q_{l}^\top}_{\widehat{S}_{k-1} F_{k-1}^{-1}} + \Delta v_{k}\Delta q_{k}^\top \right) F_{k} \nonumber \\
        & =  \left( \lambda \widehat{S}_{k-1} F_{k-1}^{-1} + \Delta v_{k}\Delta q_{k}^\top \right) F_{k} \nonumber \\
        &= \left(   \widehat{S}_{k-1}  \Big( F_k^{-1}- \Delta q_{k}\Delta q_{k}^\top \Big) + \Delta v_{k}\Delta q_{k}^\top \right) F_{k} \nonumber\\
        &=  \widehat{S}_{k-1} + \Big (\Delta v_{k} - \widehat{S}_{k-1}  \Delta q_{k}    \Big) \Delta q_{k}^\top F_k,
        \label{S_der_inter}
\end{align}
 where the next to the last equality follows from \eqref{F_update_APP}.

\bibliographystyle{IEEEtran}

\bibliography{references}

\end{document}